\documentclass[a4paper, 10pt]{amsart}
\usepackage{amsmath, latexsym,  amssymb, amscd}
\usepackage[all]{xy}
\newcommand{\emor}{{\rm End} }

\newcommand{\indice}{r}

\newcommand{\dimV}{d}

\newcommand{\qe}{\mathbb{Q}}

\newcommand{\ze}{\mathbb{Z}}

\newtheorem{thm}{Theorem}[section]
\newtheorem{con}{Conjecture}[section]
\newtheorem{propo}{Proposition}[section]
\newtheorem{lem}[propo]{Lemma}

\newtheorem{D}[propo]{Definition}

\newtheorem*{propa*}{Proposition B}
\newtheorem*{propb*}{Proposition A}

\newtheorem*{proprietas}{Property $(S^n)$}

\newtheorem{example}{Example}[section]

\title[{The optimality of the Bounded Height Conjecture } ]
{The optimality of the bounded Height Conjecture}

\author[{ Evelina Viada}]{  
 }

\begin{document}

\maketitle
\centerline{Evelina Viada\footnote{Evelina Viada,
     Universit\'e de Friboug Suisse, P\'erolles, D\'epartement de Math\'ematiques, Chemin du Mus\'ee 23, CH-1700 Fribourg,
Switzerland,
    evelina.viada@unifr.ch.}
    \footnote{Supported by the SNF (Swiss National Science Foundation).}
\footnote{Mathematics Subject classification (2000): 11G50, 14H52, 14K12.\\
Key words: Height, Elliptic curves, Subvarieties.}}

In this article we show that the Bounded Height Conjecture is optimal in the sense that, If $V$ is  an irreducible variety  in a power of an elliptic curve with empty deprived set, then all open subsets of $V$ do not have bounded height. 
The Bounded Height Conjecture is known to hold.
We also present    some examples and remarks.

\section{introduction}

This work  concerns principally the optimality of  the Bounded Height Conjecture, stated by Bombieri, Masser and Zannier \cite{BMZ1} and proven by Habegger \cite{Phil1}.  In section 2, we clarify the assumption on the varieties,
understanding such a hypothesis geometrically and from different points of view.  We  give some examples, to make sure that certain situations can occur.
In section 3, we prove the optimality of the Bounded Height Conjecture. 
In the final section  we present some further remarks and  possible open questions.

Denote by  $A$  an abelian variety    over $\overline{\qe}$  of
dimension $g$. 
Consider on $A(\overline\qe)$ a canonical height function. Denote by $||\cdot||$ the induced semi-norm.
For $\varepsilon \ge 0$, we denote $$\mathcal{O}_{\varepsilon}=\{ \xi \in A(\overline{\qe}) : ||\xi|| \le \varepsilon\}.$$

Consider a proper irreducible algebraic  subvariety   $V$ of dimension $d$ embedded in $A$, defined over $\overline{\qe}$. 
We say that: 
\begin{itemize}
 \item[-] $V$ is transverse, if $V$ is not contained in any
translate of a proper algebraic subgroup of $A$.

\item[-]  $V$ is
weak-transverse, if $V$ is not contained in  any proper algebraic
subgroup of $A$.
\end{itemize}

Given  an integer $ \indice$ with $1 \le \indice \le g$ and  a subset $F$ of $A(\overline{\qe})$, we define the set
$$S_{\indice}(V,F)= V(\overline{\qe})\cap  \bigcup_{\mathrm{cod}B \ge \indice} B+F$$
where $B$ varies over all abelian subvarieties of $A$ of
codimension at least $\indice$ and $$B+F=\{b+f \,\,\,: \,\,\,b\in B, \,\,\,f\in F\}.$$ 
Note that $$S_{r+1}(V,F)\subset S_{r}(V,F).$$
We  denote the set $S_{\indice}(V, A_{\rm Tor})$
simply  by $S_{\indice}(V)$, where $A_{\rm Tor}$ is the torsion of $A$.
For a subset  $V' \subset V$, we denote
$$S_r(V',F)=V' \cap S_r(V,F).$$
It is  natural  to ask:  `For which sets $F$ and integers
$\indice$, has the set 
 $S_{\indice}(V,F)$ bounded height or is it non-Zariski    dense in $V$?'
 
Sets of this kind, for $\indice=g$, appear in the literature in the
context of the Mordell-Lang, of the Manin-Mumford and of the Bogomolov Conjectures.
More recently  Bombieri, Masser and Zannier \cite{BMZ}  have proven that:

For a transverse
 curve $C$ in a torus,
 \begin{enumerate}
 \item The set $S_1(C)$ has bounded height, 
 \item The set $S_2(C)$ is finite. 
 \end{enumerate}
 They
investigate  for the first time,  intersections with the union of all
algebraic subgroups of a given codimension. This opens a vast
number of conjectures for subvarieties of semi-abelian varieties. 

Most naively, one could risk the following:

For $V$ a transverse subvariety of $A$,
\begin{enumerate}
\item $S_d(V)$ has bounded height,
\item $S_{d+1}(V)$ is non-Zariski dense in $V$.
\end{enumerate}
We will show that i. is a too optimistic guess.

Several problems rise for varieties. 
 A proper Zariski closed subset of a curve has bounded height. 
In general,  a proper Zariski closed subset of a variety does not have   bounded height, however it is still a `small' set. So one shall say, that outside an anomalous Zariski closed subset of $V$, the points we consider have bounded height.
 Bombieri, Masser and Zannier  introduced the anomalous set. Hardest is to show that it is closed.
 \begin{D} [\cite{BMZ1} Definition 1.1 and 1.2]
 \label{defanom}
An irreducible subvariety $X$ of $V$ is anomalous if it has positive dimension and lies in a coset $H$ of $A$ satisfying $$\dim H\le n-\dim V+\dim X-1.$$ The deprived set $V^{oa}$ is what remains of $V$ after removing all anomalous subvarieties.
\end{D}
For tori, they prove
\begin{thm}[\cite{BMZ1} Theorem 1.4.]
\label{chiuso}
The deprived set $V^{oa}$ is a Zariski open  of $V$. 
\end{thm}

Then, they state the following conjecture for tori and $\varepsilon=0$.
\begin{con}[Bounded Height Conjecture]
\label{alt1}
Let $V$ be an irreducible variety in $A$ of dimension $d$. Then, there exists $\varepsilon>0$ such that  $S_{d}(V^{oa},\mathcal{O}_\varepsilon)$ has bounded height.

\end{con}
 We remark that in all known effective proofs, the bound for the height of $S_d(V^{oa})$ is independent of the field of definition of $V$. Then, a set $F$ of bounded height does not harm.

 For transverse curves in a torus \cite{BMZ} and in a product of elliptic curves \cite{V}, Conjecture \ref{alt1} is effectively proven. In a preprint P. Habegger \cite{Phil1}  deals with subvarieties  of an abelian variety $A$  defined over the algebraic numbers. He shows:
 
 \begin{thm} [Habegger \cite{Phil1}]
 \label{phill}
 For $V$  an irreducible subvariety of $ A$, Conjecture \ref{alt1} holds.
 \end{thm}

 In the first instance we analyze
  several geometric properties which are different for varieties, but they all collapse to the transversal condition  for curves.
  \begin{proprietas}
    We say that $V$ satisfies  Property $(S^n)$ if,
 for all morphism $\phi: A \to A$ such that $\dim \phi(A)\ge d+n$,  $$\dim \phi(V)=\dimV.$$
 We simply say Property $(S)$ for $(S^0)$.
  \end{proprietas}
  In some sense Property $(S)$ is natural.   Property $(S^n)$ implies Property $(S^{n+1})$ and also implies transversality. For curves, transverse implies Property $(S)$.
 
 Habegger and R\'emond (see lemma \ref{equi}) show that property (S) is equivalent to the assumption $V^{oa}\not=\emptyset$. 
  Then, one can easily reformulate the Bonded Height Conjecture in terms of Property $(S)$,  avoiding the notion of deprived set.
      \begin{con} [Bounded Height Conjecture]
 \label{alt}
Let  $V$ be an  algebraic subvariety of $A$  defined over $\overline{\qe}$.  Suppose that $V$ satisfies Property $(S)$.
 Then, there exists  $\varepsilon>0$ and a non-empty open subset $V^e$ of $V$ 
  such that
  $S_{d}(V^{e},\mathcal{O}_\varepsilon)$  has bounded height.

 \end{con}

One could hope to relax the assumption of Property $S$ on the variety.  Could it be sufficient to assume, as we do for curves, that  $V$ is transverse? What about a product of varieties which do satisfy Property $S$? 
 In section 3, we prove that theorem \ref{phill} is optimal for subvarieties of a power of an elliptic curves $E^g$.
\begin{thm}
\label{sottimale}
Let $V$ be a subvariety of $E^g$ of dimension $d$. Suppose that $V$ does not satisfy Property $(S)$ (or equivalently that $V^{oa}=\emptyset$).
Then, for every non-empty Zariski open subset $U$ of $V$ the set $S_d(U)$ does not have bounded height.

\end{thm}

The proof is constructive. A fundamental point is to associate to a non-torsion point of  $E(\overline\qe)$ a Zariski dense subgroup of $E^n$.

A natural rising question is to investigate the height for larger codimension of the algebraic subgroup.
Let $\Gamma$ be a subgroup of $A(\overline{\mathbb{Q}})$ of finite rank. We denote $\Gamma_\varepsilon=\Gamma+\mathcal{O}_\varepsilon$. 

 \begin{con}
\label{hb}

Let $V$ be an irreducible algebraic subvariety  of $A$ of dimension $d$, defined over $\overline{\qe}$. 
 Then there exists  $\varepsilon>0$
  and a non-empty Zariski open subset $V^e$ of $V$ such that:
\begin{enumerate}
\item  If $V$ is weak-transverse, 
  $S_{d+1}(V^e,\mathcal{O}_\varepsilon)$  has bounded height.

\item  If $V$ is transverse,   $S_{d+1}(V^e,\Gamma_\varepsilon)$ has bounded height.
\end{enumerate}
\end{con}
In some cases Conjecture \ref{hb} is proven. 
For $\Gamma \not=0$ or $V$ weak-transverse but not transverse, the method used for the proofs is  based on a Vojta inequality. This 
   method is not effective. It  gives
optimal results for curves (see \cite{RV} Theorem 1.5 and  \cite{io} Theorem 1.2). On the contrary, for varieties of dimension at least two a hypothesis stronger than transversality is needed. Part i. of the following theorem  is proven  by  R\'emond \cite{RHG2}   Theorem 1.2 and \cite{Gprep}. Whereas, Part ii. is proven by the author  \cite{io2} Theorem 1.5.

\begin{thm}
%[Part i. R\'emond, Part ii. Viada]
\label{wtalt}
Let $V$ be an irreducible subvariety  of $E^g$ of dimension $d$, defined over $\overline{\qe}$.  Let $p$  be a point in $E^s(\overline{\qe})$ not lying in any proper algebraic
subgroup of $E^s$.
Assume that $V$ satisfies 
\begin{equation}
\label{stella}
\dim (V+B)=\min(\dim V+\dim B, \,\,\, g )
\end{equation} for all abelian subvarieties $B$ of $E^g$.    Then there exists a non-empty Zariski open subset $V^e$ of $V$ and $\varepsilon>0$ such that:
\begin{enumerate}
\item $S_{d+1}(V^e,\Gamma_\varepsilon)$ has bounded height,
\item $S_{d+1}(V^e\times p,\mathcal{O}_\varepsilon)$ has bounded height.
\end{enumerate}

\end{thm}

%A  weak-transverse variety in $E^n$ is isogenous to $V\times p$ for $V$ transverse in some $E^g$ and $p$  a point in $E^{n-g}$ not lying in any proper algebraic subgroup (see \cite{io2} section 3).

In Lemma \ref{ps}, we will see that the assumption (\ref{stella}) is equivalent to Property $(S)$.
Finally we give some  examples of varieties satisfying Property $(S)$ and of varieties which do not satisfy Property $(S)$ but for which Conjecture \ref{hb} holds. 

To conclude we remark that, if one knows that, for $r\ge d+1$ and $V$ transverse,  the set  $S_r(V^e,\Gamma_\varepsilon)$ has bounded height,  then \cite{io2}  Theorem 1.1 implies that $S_r(V^e, \Gamma_\varepsilon)$ is not Zariski dense in $V$. If $\Gamma$ has trivial rank, it is sufficient to assume $V$ weak-transverse. This makes results on heights particularly interesting. \\

{\it Acknowledgments:} 
I kindly thank the Referee for his accurate and nice suggestions.

 \newpage

 \section{preliminaries}

\label{nota}

Let $E$ be an elliptic curve defined over a number field. 
All statement in the introduction become trivially verified  for a zero-dimensional variety. In the following we avoid this case.
Let $V$ be an irreducible algebraic subvariety of $E^g$ of dimension $0<d<g$ defined over $\overline{\qe}$.

We fix on $E(\overline{\qe})$ the canonical N\'eron-Tate height function. We denote by   $||\cdot||$ the induced semi-norm on $E(\overline\qe)$.   For  $x=(x_1, \dots,x_g)\in E^g(\overline{\qe})$, we denote  $$||x||=\max_i||x_i||.$$ 
For $\varepsilon \ge 0$, we define
$$\mathcal{O}_{\varepsilon}=\{ \xi \in E^g(\overline{\qe}) : ||\xi|| \le \varepsilon\}.$$ 
The height  of a non-empty set $S\subset E^g(\overline{\qe})$ is the supremum of the heights  of
its elements. 
The degree of $S$ is the degree (possibly $\infty)$ of the field of definition of the points of $S$.

The ring of endomorphism $\emor(E)$ is isomorphic  either to $\ze$ (if $E$ does not have C.M.) or to an order in an imaginary quadratic field  (if $E$ has C.M.). 
 We consider on $\emor(E)$ the hermitian scalar product $\langle \cdot,\cdot \rangle$ induced by $\mathbb{C}$ and denote by $|\cdot|$ the associated norm. Note that the metric does not depend on the   embedding of $\emor(E)$ in $\mathbb{C}$.

We denote by $M_{r,g}(\emor(E))$ the  module of $r\times g$ matrices with
entries in  $\emor(E)$.
For $F=(f_{ij})\in M_{r,g}(\emor(E))$,  we define $$|F|=\max_{ij}|f_{ij}|.$$

We identify a morphism $\phi:E^g \to E^r$ with a matrix  in
$M_{r,g}({\rm End }(E))$.

Let $B$ be an algebraic subgroup of $E^g$ of codimension $r$. Then $B \subset \ker \phi_B$ for a  surjective morphism $\phi_B:E^g\to E^r$.
 Conversely, we denote by $B_\phi$ the kernel of a surjective morphism $\phi:E^g \to E^r$. Then $B_\phi$  is an algebraic subgroup of $E^g$ of codimension $r$.

If $\phi:E^g \to E^{g'}$ is a surjective morphism, we can complement $\phi$ and define an isogeny $f:E^g \to E^g$ such that  $f(\ker \phi)=0\times E^{g-g'}$ and  $\pi_1f=\phi$, where $\pi_1:E^{g} \to E^{g'}$ is  the natural projection on the first $g'$ coordinate.
More precisely; recall that every abelian subvariety of $E^g$ of
dimension $n$ is isogenous to $E^n$. Then $\ker \phi$ is isogenous to $E^{g-g'}$, let $i$ be such an isogeny.
Let $ (\ker \phi)^{\perp}$ be an orthogonal complement of $\ker \phi $ in $E^g$. Then $E^{g'}$ is isogenous to $ (\ker \phi)^{\perp}$. Let $j:E^{g'} \to (\ker \phi)^{\perp}$ be such an isogeny.  Define  
 the isogeny 
 \begin{equation*}
 \begin{split}
 f: &E^g \to E^g\\ 
 &x \to \left(\phi(x), i(x-j(\phi(x))\right) .
 \end{split}
  \end{equation*}
This  $f$ has  the wished  property.

Let us state a classical:
\begin{lem}
\label{projezione}
For every algebraic subvariety $X$ of $E^g$ of dimension $d$ there exists a
projection on $d$ coordinates such that the restriction to $X$ is dominant.
\end{lem}

\begin{proof}

Let $d_0$ be the maximal integer such that the restriction of
$\pi_0: E^g\to E^{d_0}$ to $X$   is surjective.
If $d_0\ge d$, nothing has to be shown. Suppose that 
$
d_0<d.$
Without loss of generality, suppose that $\pi_0$ projects on the first
$d_0$ coordinates. For $d_0<i\le g$, we define  $\pi_i:E^g\to
E^{d_0+1}$ to be the projection $\pi_i(x_1,\dots,x_g)\to (x_1, \dots,
x_{d_0},x_i).$ Let $i_\pi:E^{d_0+1}\to E^g$ be the immersion such that $\pi_i\cdot i_\pi = id_{E^{d_0+1}}$. 

We denote by $X_i=i_\pi \cdot\pi_i(X)\subset E^g$. By maximality of $d_0$ we see that $\dim
X_i=d_0$. Furthermore
$X$ is the fiber product of $X_i$ over $\pi_0(X)=\pi_0(X_i)$. Then $d=\dim X=\dim (X_{d_0+1}\times_{\pi_0(X)} \dots \times_{\pi_0(X)}X_g)=d_0$, which contradicts $d>d_0$.

\end{proof}

We  show an easy  application.

\begin{lem}
\label{duno}
If $V$ does not satisfy Property $(S)$ then there exists a surjective morphism $\phi:E^g \to E^d$ such that $0<\dim \phi(V)<d$.
\end{lem}
\begin{proof}
If $V$ does not satisfy Property $(S)$, then there exists a surjective morphism $\phi:E^g \to E^d$ such that $\dim \phi(V)<d$. If $\dim \phi(V)>0$, nothing has to be shown. If $\dim \phi(V)=0$,  Lemma \ref{projezione} gives a morphism $r:E^g \to E$ such that the restriction to $X$  is surjective. Replace the first row of $\phi $ by $r$.
\end{proof}

 \section{The Bounded Height Conjecture and its optimality}
 \label{heights}
In the following we first  show that the set $S_d(V)$ is  dense in $V$.

We then ask if  Property $(S)$ is necessary to show that $S_d(V)$ has bounded height.
We give here a positive answer. Meanwhile we  try to understand  the geometric aspect of
 Property $(S)$.

An easy example of a variety which does not satisfy Property $(S)$ is a split variety
$V_1 \times V_2 \times \dots \times V_n$ with the $V_i \subset E^{g_i}$.
It is  natural to ask  if only this kind of product varieties  do not satisfying Property $(S)$.
This is not the case, as Lemma \ref{gsplit}  and  Example \ref{ex3} show.
\begin{D}

Let $V\subset E^g$ be a variety of dimension $d$.
\begin{enumerate}

\item $V$ is split if there exists an isogeny
$\phi:E^g \to E^g$ such that $$\phi(V)=V_1 \times V_2$$ with $V_i\subset E^{g_i}$ and $g_i\not=0$, for $i=1,2$.

\noindent We say that $V$ is non-split if the above property is not verified. 

\item  $V$  is $n$-generically split if  there exists an isogeny
$\phi:E^g \to E^g$ such that  $\phi(V)$ is contained in a proper
split variety $W=W_1\times W_2$ with $W_i \subset E^{g_i}$ and 
$$\dim W_1<\min(d,g_1-n).$$
% $W_1$  has dimension $d'<d$  and $W_i\subset E^{g_i}$. %with $g_i\not=0$.

\noindent We say that $V$ is $n$-generically non-split if it is not $n$-generically split.

\noindent We simply say generically split for $0$-generically split.
\end{enumerate}
\end{D}
Clearly generically non-split implies non-split.
Note that non-split implies
 transverse. Indeed if $V$ is not transverse, then there exists an isogeny $\phi:E^g\to E^g$ such that $\phi(V)\subset p \times E^r $. Set $V_1=p$ and $V_2= \pi(\phi(V))$, where $\pi$ is the projection on the last $r$ coordinates.

The following Lemma clarifies  the equivalence  between Property $(S^n)$ and the   $n$-generically non-split property.

\begin{lem}
\label{gsplit}
A subvariety $V\subset E^g$ satisfies Property $(S^n)$ if and only if $V$ is $n$-generically non-split.
\end{lem}

\begin{proof}
First suppose that $V$ does not satisfy Property $(S^n)$.  Then, there exists 
$\phi_1:E^g\to E^{d+n}$  such that $V_1=\phi_1(V)$ has dimension $d_1<d$. 
Let $f={\phi_1 \choose \phi_1^\perp}$. 
Then $$f(V)\subset V_1\times E^{g-d-n}.$$ 
 Furthermore $\dim V_1<d=\min(d,d+n-n)$. Thus $V $ is $n$-generically split.

Secondly suppose $V$ is $n$-generically non-split. Then, up to an isogeny,  $V$ is contained in $W=W_1\times W_2$ with $W_i \subset E^{g_i}$ and $ \dim W_1=d_1<\min(d,g_1-n)$. 
Consider the  projection $V_1$ of $V$ on  the first $d+n$ coordinates.

 If $g_1\ge d+n$, then $V_1$ is contained in the projection of $W_1$. As dimensions cannot increase by projection, $\dim V_1\le \dim W_1<d$.

 If $g_1<d+n$, then  we have $V_1\subset W_1 \times E^{d+n-g_1}$. Thus $\dim V_1 \le d_1+d+n-g_1<d$ because $d_1<g_1-n$.
 So $V$ does not satisfy Property $(S^n)$.
 \end{proof}
 
 It is then natural to give an example of a non-split variety which is  generically split, or equivalently which does not satisfy Property $(S)$.

\begin{example}
Let us show at once that for a hypersurface, the notion of non-split and generically non-split coincide.  

Let $V$ be a non-split hypersurface in $E^{d+1}$. If $V$ were
generically split, then, for an isogeny $\phi$, $\phi(V)$ would be  contained
in a proper split variety $
W_1\times W_2$. For a dimensional argument $\phi(V)=W_1\times W_2$,
contradicting the non-split assumption.
\end{example}
\begin{example}
In some sense,  to give an example of a non-split but generically-split   variety it is necessary to consider varieties of large codimension. 

In $\mathbb{G}_m^n$ it is easier to write equations.
Consider the surface $V$ in $\mathbb{G}_m^4$  parameterized by $u$ and $v$, and given by the set of points $(u, u^5+1,5u^4v+u, v+u^5+1)$. This is simply the envelope variety $V$ of the irreducible plane curve $C=(u,u^5+1)$. The  envelope is constructed as follows. To a point  $p\in C$  we associate  the tangent line $t_p$ in p.  Then $V=\cup_{p\in C}(p,t_p)$. 
A property of the envelope is that it is not the fiber product  of two varieties of positive dimension. 

%Indeed, it is easy to see that it is transverse. In addition, no linear transformation from $(u,v)$ to $(u',v')$ can bring $(u, u^5+1,5u^4v+u, v+u^5+1)$ in polinomials of the form $(f_1(u'),f_2(u'),g_1(v'),g_2(v'))$.

The set $V$ is an 
algebraic surface; let $z_1, z_2,z_3,z_4$ be the variables, then $V$
is the zero set of 
\begin{equation*}
\begin{cases}
z_3=5z_1^4(z_4-z_2)+z_1,\\ z_2 =z_1^5+1.
\end{cases}
\end{equation*}

 The projection on the first two coordinates  is exactly the  curve  $C$ defined by $z_2=z_1^5+1$. Thus $V$ does not satisfy Property $(S)$, however it is non-split. If, on the contrary, $V$ were split, then, for an isogeny $\phi$, $\phi(V)=V_1\times V_2$.  Since $V$ is transverse, $V_i$ have positive dimension. In addition $\ker \phi \cdot V=\phi^{-1}(V_1)\times_{\ker\phi} \phi^{-1}(V_2)$. Thus $V=W_1\times_{ W_1\cap W_2}  W_2$ where $W_i=V\cap \phi^{-1}(V_i)$.  This contradicts  that $V$ is not a fiber product.
\end{example}
\begin{example}
\label{ex3}
We now extend the previous  example to a power of an elliptic curve
$E$ given by a Weierstrass equation in $\mathbb{P}^2$.
Consider the projection from $E \to \mathbb{P}^1$ given by the projection of a point $(v_1:v_2:1) \in E$ on  the first
coordinate  $(v_1:1)$. Do the same on each factor to have a projection from $E^4$ to $
\left(\mathbb{P}^1\right)^4$.  The  torus $\mathbb{G}_m^4$ is naturally an open in $
\left(\mathbb{P}^1\right)^4$.

Consider in
$\mathbb{G}_m^4$ the non-split but generically split $V$ just constructed above. Take the
preimage  of $V$ on $E^4$ and its closure $V'$. Then
$V'$ is generically split, however it is non-split. As above, suppose $\phi(V')=V'_1\times V'_2$ for an isogeny $\phi$. Then  the preimage of $\phi^{-1}\phi (V')$ to $(\mathbb{P}^1)^4$ were a fiber product, contradicting that $V$ is not a fiber product.

\end{example}

We remark, that there are also non-split transverse varieties which do not satisfy Property $(S^n)$: One can extend this last example   taking the envelope surface of a transverse curve in  $E^{n+2}$.

 R\'emond \cite{Gprep} theorem 1.9 and Habegger \cite{Phil1} corollary 2 prove that if $V^{oa}=\emptyset$ then $V$ does not satisfy property $(S)$. Using the generically-split  property we prove the reverse implication.
\begin{lem}
\label{equi}
A variety $V$ does not satisfy property $(S)$ if and only if $V^{oa}= \emptyset$
\end{lem}
\begin{proof}
Suppose that $V$  has dimension $d$ and does not satisfy property $(S)$. 
By lemma \ref{gsplit}, there exists an isogeny $\phi$ such that $\phi(V)\subset W_1\times W_2$ with $W_i\subset E^{g_i}$ and  $\dim W_1 <\min (d,g_1)$. Then, the  intersection of $V$ with the cosets $\phi^{-1}(x\times E^{g_2})$ for $x\in W_1$ are either empty of anomalous. In addition each point of $V$ belongs to such an intersection. So $V^{oa}$ is empty. The reverse implication is proven by Habegger \cite{Phil1} corollary 2 using R\'emond \cite{Gprep} theorem 1.9.

%On the other hand, if $V^{oa}=\emptyset$, there exists a  coset of an abelian subvariety $B$ of codimension $1<h<g$ such that
%$B\cap V$ has positive dimension. Let $\pi_B$ be the projection  to
%the  orthogonal complement of $B$. If $h\ge \dim V$, a further
%projection to $E^{\dim V}$ breaks property $(S)$. If $h< \dim V$, the
  %direct sum of $\pi_B$ and the projection on the first $\dim V-h$
 %factors breaks property $(S)$.
\end{proof}

The following lemma shows that in the Bounded Height Conjecture we can not expect the set in the consequence to be non-dense. This lemma will also be used in the proof of Theorem \ref{sottimale}.
\begin{lem}
\label{s1denso}
Let $V$ be an irreducible subvariety of $E^g$ of dimension $1\le d<g$. Then
the set
$S_d(V)\setminus S_g(V)$ is dense in $V$. 

\end{lem}
\begin{proof}
 We shall distinguish two cases with regard to whether  $V$ is or not the translate of an
 abelian subvariety by a torsion point.

Suppose $V$ is not such  a translate. Then, the Manin-Munford Conjecture, a
theorem of Raynaud, ensures that the torsion $S_g(V)$ is not dense in $V$. Our
claim is then equivalent to show that $S_d(V)$ is dense in $V$. Consider a surjective morphism (for example a projection) $\phi:E^g\to E^d$ such that the restriction to $V$ is dominant. Use lemma \ref{projezione}  to ensure the existence of such a morphism.
 Let   $E^d_{\rm{Tor}}$ be the torsion group of $E^d$.
 The preimage on $V$ via $\phi$ of  $E^d_{\rm{Tor}}$ is dense in $V$
 and it  is a subset of $S_d(V)$.

Suppose now that $V$ is the translate of an abelian subvariety by a torsion point. Up to
an isogeny, we can assume $V=E^d\times p$ for $p=(p_1,\dots,p_{g-d})\in
E^{g-d}_{\rm Tor}$. Note that, by Kronecker's Theorem, for any $x\in V$, $x+E^d_{\rm Tor}$ is dense in $V$.

Since $p$ is a torsion point,  $$\left(\left(E(\overline{\qe})\setminus E_{\rm
    Tor}\times\{0\}^{d-1}\times p\right) +E^d_{\rm Tor}\right) \subset
S_{d}(V)\setminus S_g(V).$$ In addition this set is  dense in $V$,
because $E(\overline{\qe})\setminus E_{\rm Tor}$ is non-empty (even dense in $E$).

%If $p$ is not a torsion point,  we can suppose,  for instance, that $p_1$ is non-torsion.
%Define
%$\phi=\begin{pmatrix} 1 &0&1 &0\dots 0\\
%0 &I_{d-1}&0& 0\dots 0
%\end{pmatrix}
%$ where $I_{d-1}$ is the identity matrix.
%Then $$\left(V\cap \ker \phi\right) +E^d_{\rm Tor}\subset
%S_{d}(V)\setminus S_g(V).$$ In addition this set is dense in $V$, because the point $(-p_1, 0,\dots,0,p)\in V\cap \ker \phi$.
\end{proof}

We now discuss the assumption of Property $(S)$.
In general, for $V=V_1\times V_2$ with $\dim V_1=d_1$ and $\dim V_2=d_2$, we have  $S_{d_1}(V_1)\times
S_{d_2}(V_2)\subset S_d(V)$. 
Could we have equality if we assume, for example, that each factor satisfies Property $(S)$?
Similarly, does Conjecture \ref{alt} hold for such a product variety or for a non-split variety?
The answer is negative. 

To simplify  the formulation of the statements we  characterize  the  sets which break  Conjecture \ref{alt}.

\begin{D}
\label{DU}
 We say that a subset $V^u$ of $V(\overline{\qe})$ is
densely unbounded if  $V^u$ is Zariski dense in $V$ and for every non-empty Zariski open  $U$ of  $V$  the intersection
$V^u\cap U$ does not have bounded height. Equivalently $V^u$ is densely unbounded if, for a sequence $\{N\}$ of positive reals  going to infinity, the set 
$$V^u[N]=\{x\in V^u \,\,\,:\,\,\, ||x||>N\}$$
is Zariski dense in $V$.
\end{D}
\begin{proof}[ Proof of the equivalence of the definitions]
Suppose that there exists a non-empty open $U$ such that $V^u\cap U$ has height bounded by $N_0$. The set  $Z=V\setminus U$ is  a proper closed subset of $V$, and therefore not dense. So $V^u[N_0+1]\subset Z$  can not be dense.

Suppose now that there exists an unbounded sequence $\{N\}$ such that $V^u[N_0]$ is not dense for some $N_0$. Then the Zariski closure $Z$ of $V^u[N_0]$ is a proper closed subset of $V$. So $U=V\setminus Z$ is a non-empty open set such that  $V^u \cap U$ has height bounded by $N_0$.
\end{proof}

Let us prove a preparatory lemma for the proof of Theorem \ref{sottimale}.
\begin{lem}
\label{gammadenso}
 
 Let $z_0$ be a non-torsion point in $E(\overline{\qe})$. Let $n$ be a positive integer.  
 Define $G_{z_0,n} =\langle z_0\rangle^n_{{\rm {End}}(E)}$.
 For $N\in \mathbb{N}$, the set 
$$G_{z_0,n}[N||z_0||]=\{p\in G_{z_0,n}\,\,\,:\,\,\, ||p||> N||z_0||\}$$
is Zariski dense in $E^n$.
As a consequence  $G_{z_0,n} $ is dense in $E^n$.

\end{lem}
\begin{proof}
Denote by $\Sigma=\langle z_0 \rangle_{{\rm End}(E)}$ the submodule of $E$ generated by $z_0$. Then $G_{z_0,n} =\Sigma^n$. Recall that $\Sigma[N||z_0||]=\{p\in \Sigma \,\,\,:\,\,\,||p||>N||z_0||\}$. Then $(\Sigma[N||z_0||])^n\subset G_{z_0,n}[N||z_0||]$. As $\Sigma[N||z_0||]$ is an infinite set, it is dense in $E$. Then $(\Sigma[N||z_0||])^n$ is dense in $E^n$.

Note that  $G_{z_0,n} $ contains $G_{z_0,n}[0]$, so it is also dense.

\end{proof}
We are ready to  show the optimality of the Bounded Height Conjecture.

\begin{proof}[{\bf Proof of Theorem \ref{sottimale}}]

Suppose that $V$ 
does not satisfy Property $(S)$. We are going to construct a densely unbounded set of $V$ which is a subset of $S_d(V)$.

By Lemma \ref{duno},  there exists   a surjective morphism $\psi:E^g \to E^{\dimV}$  such that $0<\dim \psi(V)<d$. Denote $V_1=\psi(V)$ and $d_1=\dim V_1$.
We can fix an isogeny and suppose that $\psi$ is the projection on the first $\dimV$ coordinates, thus $V\subset V_1\times E^{g-\dimV}$. 
Let  $x \in V$. Then $x=(x_1, x_2)$ with $x_1\in V_1$ and $x_2 \in E^{g-\dimV}$. Consider $$x_1\times W_{x_1}=V\cap (x_1\times E^{g-\dimV}).$$ 
 There exists an open dense subset $U_1$ of $V_1$ such that the algebraic variety  $W_{x_1}$ is equidimensional of   dimension $\dimV_2=\dimV-\dimV_1$. Let $ V_{x_1}$ be an irreducible component of $ W_{x_1}$. By Lemma \ref{projezione}, 
there exists a projection $\pi_{x_1}:E^{g-\dimV}\to E^{\dimV_2}$ such that the restriction 
\begin{equation}
\label{dom}
{\pi_{x_1}}_{|  V_{x_1}}:  V_{x_1} \to E^{\dimV_2}
\end{equation} is dominant and  even surjective and therefore its fibers are generically finite.

Consider $V_1 \subset E^{\dimV}$.  Since $V$ is irreducible also $V_1$
is. 
By Lemma \ref{s1denso}, applied with $V=V_1$, $d=d_1$ and $g=d$, the set  $ S_{d_1}(V_1){\setminus S_{d}(V_1)}$ is  Zariski dense in $V_1$. 
Define $$V_1^u=U_1 \cap \left( S_{\dimV_1}(V_1){\setminus S_{d}(V_1)}\right).$$
Then all points in $V_1^u$ are non-torsion and $V_1^u$ is a dense subset of $V_1$.
By definition of  $S_{d_1}(V_1)$, if $x_1 \in V_1^u\subset S_{d_1}(V_1)$, then there exists $\phi_1:E^{\dimV}\to E^{\dimV_1}$ of rank $d_1$ such that 
\begin{equation}
\label{fuz}
\phi_1(x_1)=0.
\end{equation}
Let $z_k$ be a  coordinate of $x_1=(z_1,\dots ,z_d)$ such that $||z_k||=\max_i||z_i||$. Only the torsion has norm zero. Since $x_1$ is non-torsion, then $||z_k||>0$.

For each point $x_1 \in V_1^u$ we will construct a subset of $x_1\times  V_{x_1}$  which is, both, densely unbounded in $x_1\times  V_{x_1}$ and  a subset of $S_d(V)$.

We denote by \begin{equation*}
 \phi_2=(0,\dots,0,\varphi_k,0,\dots,0):E^{\dimV}\to E^{\dimV_2}
 \end{equation*}  a morphism such that only the $k$-th column is non zero.
 
 For a positive integer $N$, we define
 $$\mathcal{F}(N):=\{\phi_2=(0,\dots,0,\varphi_k,0,\dots,0):E^{\dimV}\to E^{\dimV_2} \,\,\,{\rm s.t.}\,\,\,|\phi_2|>N\}$$
and
$$x_1\times  V_{x_1}^{u}(N):=\{(x_1,y)\in (x_1, V_{x_1})\,\,\,{\rm s.t.}\,\,\,\exists \,\,\,\phi_2 \in \mathcal{F}(N) \,\,\,{\rm with\,\,\,} \phi_2(x_1)=\pi_{x_1}(y)\}.$$
We simply denote $$  U_{x_1}=  V_{x_1}^{u}(1).$$
We want to show that $ x_1\times  U_{x_1}$ is densely unbounded in $x_1\times  V_{x_1}$.

(a) - First we show that 
$$x_1\times  V_{x_1}^{u}(N)\subset ( x_1\times  U_{x_1} )[N||z_k||].$$
For $(x_1,y)\in x_1\times  V_{x_1}^{u}(N)$ there exists $\phi_2\in \mathcal{F}(N)$ such that $\phi_2(x_1)=\pi_{x_1}(y)$. Thus,
$$||y||\ge||\pi_{x_1}(y)||=||\phi_2(x_1)||\ge |\phi_2|||z_k||> N||z_k||.$$  
Whence $(x_1,y) \in( x_1\times  U_{x_1} )[N||z_k||]$.

(b) - We now show that $x_1\times  V_{x_1}^{u}(N)$ is dense in $x_1\times  V_{x_1}$. Let $(a_1z_k, \dots, a_{d_2}z_k)\in G_{z_k, d_2}[N||z_k||]$ with     $a_i\in \emor(E)$. Then $\max_i|a_i|>N$. Let $\phi_2$ be the morphism from $E^d$ to $E^{d_2}$ such that the $k$-th column of $\phi_2$ is the vector  $\varphi_k=(a_1,\dots, a_{d_2})^{\rm{t}}$ and all other entries are zeros. Then $\phi_2 \in \mathcal{F}(N)$ and $\phi_2(x_1)=(a_1z_k, \dots, a_{d_2}z_k)$.    
    So, we have the inclusion
  $$G_{z_k, d_2}[N||z_k||]\subset \bigcup_{\phi_2\in \mathcal{F}(N)}{\phi}_2(x_1).$$
  By Lemma \ref{gammadenso},  $G_{z_k,d_2}[N||z_k||]$  is a Zariski dense subset of $E^{\dimV_2}$.   Thus, also   the set $\bigcup_{\phi_2\in \mathcal{F}(N)}{\phi}_2(x_1)$ is Zariski dense in $E^{\dimV_2}$.
By (\ref{dom}) the map  ${\pi_{x_1}}_{| V_{x_1}}$ is surjective. Then for any ${\phi_2\in \mathcal{F}(N)}$ there exists $y\in  V_{x_1}$ such that $\pi_{x_1}(y)=\phi_2(x_1)$. Therefore $ x_1\times  V_{x_1}^u(N)$ is Zariski dense in
  $x_1\times  V_{x_1}$. 

In view of Definition \ref{DU}, part (a) and (b) above show that  $ x_1\times  U_{x_1} $ is a densely unbounded subset  of   $x_1\times  V_{x_1}$. In addition, by definition of $x_1\times U_{x_1}$,    for every $(x_1,y)\in  x_1\times  U_{x_1}$ there exists 
$\phi_2: E^d \to E^{\dimV_2}$ such that 
\begin{equation}
\label{B}
\phi_2(x_1)=\pi_{x_1}(y).
\end{equation}

Consider   $(x_1,y)$ with $x_1\in V_1^u$ and $y \in  U_{x_1}$.  By relations (\ref{fuz}) and (\ref{B}),  the  morphism 
\begin{equation*}
%\label{qui}
\phi=\left(
\begin{array}{cc}
\phi_1 &0\\
-\phi_2 &\pi_{x_1}
\end{array}
\right):E^g \to E^d,
\end{equation*}
has rank equal to ${\rm{rk}\,\,} \phi_1+{\rm{rk}\,\,} \pi_{x_1}=d_1+d_2$ and $$\phi(x_1,y)=0.$$
So $(x_1,y) \in S_d(V)$.

 Let $$ W_{x_1}^u=\bigcup  V^u_{x_1}(1)=\bigcup U_{x_1}$$ for  $V_{x_1}$ varying over the irreducible components of  $W_{x_1}$.
We conclude that the set $$\bigcup_{x_1\in V_1^u
 } x_1\times W^u_{x_1}\subset
S_d(V)$$ is densely unbounded  in $V$.

\end{proof}

\section{Final Remarks}

It is  then natural to investigate the height property for the codimension of the algebraic subgroups at least $d+1$. We expect that Conjecture \ref{hb} holds.
Let us say at once that the  (weak)-transverse hypothesis is in general necessary, however it is not clear if it is sufficient.

Theorem \ref{wtalt} is a special case of Conjecture \ref{hb}. We show that the 
 condition  (\ref{stella}) coincides with Property $(S)$. Compare the following lemma with \cite{RHG2} lemma 7.2.
\begin{lem}
\label{ps}
An irreducible  variety $V\subset E^g$ satisfies Property $(S)$ if and only if $\dim (V+B)=\min(\dim V+\dim B, g )$ for all abelian subvarieties $B$ of $E^g$.

\end{lem}
\begin{proof}
Note that $E^g/B$ is isogenous to $E^{g-\dim B}$. Consider the natural projection $\pi_B: E^g \to E^{g-\dim B}$. Then
\begin{equation}
\label{uf}
\dim \pi_B(V)= \dim (V+B) -\dim B.
\end{equation}
Denote by $d$ the dimension of $V$. Suppose that $V$ satisfies Property $(S)$, we have
\begin{enumerate}
\item[-] If $g-\dim B\ge \dim X $, then $\dim \pi_B(X)=d$.
\item[-] If $g-\dim B\le \dim X $,  then $\dim \pi_B(X)=g-\dim B$.
\end{enumerate}
Use (\ref{uf}) to deduce $\dim (V+B)=\min(d+\dim B,  g )$.

Suppose now that $\dim (V+B)=\min(d+\dim B, g )$ for all abelian subvarieties $B$ of codimension  $d$. Note that, if $\phi:E^g \to E^d$ is a surjective morphism, then the zero component of $\ker \phi$ is an abelian variety of codimension $d$. Relation (\ref{uf}) show at once that $V$ satisfy Property $(S)$.
\end{proof}

We observe that,   for $S_{d+1}(V)$, the natural analogue to Conjecture
\ref{alt}, is to assume 
Property $(S^1)$. 
Property $(S^1)$ is weaker than $(S)$. There are even split varieties which
satisfy Property $(S^1)$.

Potentially, the method used by Habegger to prove Theorem  \ref{phill}, extends to show that, for 
$V$ satisfying Property $(S^1)$, there exists a non-empty open $V^e$ such that $S_{d+1}(V^e)$ has bounded height. 

However, neither such a  statement nor Theorem \ref{wtalt}  are  optimal: transversality is expected to be a sufficient assumption, as the following examples suggest.
We  give  simple examples of a transverse variety $V$ of dimension $d$ which does not satisfy Property $(S)$ or $(S^1)$ but such that $S_{d+1}(V,\Gamma_\varepsilon)$ is non-Zariski dense.

\begin{example}
Let $V_1$ be a variety in $E^{d_1+n+1}$ of dimension $d_1$. Suppose that $V_1$ satisfies Property $(S)$. If you like take a transverse curve. By Theorem \ref{wtalt} i., for every $\Gamma'$ of finite rank there exists $\varepsilon>0$ such that $S_{d_1+1}(V_1, \Gamma'_\varepsilon)$ has bounded height.
By  \cite{io2}  Theorem 1.1, applied to $V_1$ of dimension $d_1$, we obtain  that  there exists $\varepsilon>0$ such that:
\begin{itemize}
\item[(1)]
$S_{d_1+1}(V_1,\Gamma'_\varepsilon)$ is non-Zariski dense in 
$V_1$. 
\end{itemize}
Let $V=V_1\times E^{d_2}$ and $g=d_1+d_2+n+1$, then $V$ is transverse in $E^g$. Furthermore,
$V$ does not satisfy Property $(S^n)$. Indeed $\dim V=d=d_1+d_2$. The projection on the
first $d+n$ coordinates is $V_1\times E^{d_2-1}$ which has
dimension  $d-1$.  Let $\Gamma\subset E^g$ be a subgroup of finite rank and let $\Gamma'$ be its projection on the first $d_1+n+1$ coordinates.
 
 By  \cite{io2} Lemma 4.1 we obtain
 $$S_{d+1}(V,\Gamma_\varepsilon)\subset S_{d_1+1}(V_1,\Gamma'_\varepsilon)\times E^{d_2}.$$  
Then,  using relation (1) above,
 $S_{d+1}(V,\Gamma_\varepsilon)$ is  non-Zariski dense in $V$. Define $Z=S_{d+1}(V,\Gamma_\varepsilon)$. Then $S_{d+1}(V \setminus Z,\Gamma_\varepsilon)$ is empty and so it also has bounded height.
\end{example}

\begin{example}
Let $V=V_1 \times V_2$ with $V_i$ a hypersurface in $E^{d_i+1} $ satisfying Property $(S)$. 
The projection on the first  $d=d_1+d_2$ coordinates shows that  $V$ does not satisfy Property $(S)$. However  $V$ satisfies Property $(S^1)$. 
 We are going to show that  
\begin{equation}
\label{unisce}
\begin{split}
S_{d+1}(V,F)\subset & \quad S_{d_1}(V_1,F)\times S_{d_2}(V_2,F)\\&\cup \big( S_{d_1+1}(V_1,F) \times V_2 \big)\\&
\cup \big(V_1\times S_{d_2+1}(V_2,F)\big).
\end{split}
\end{equation} 

Let $g=d_1+d_2+2$.
Let $(x,y)\in S_{d+1}(V,F)$ with $x\in V_1$ and $y\in
V_2$. Then, there exist $\phi:E^g \to E^{d+1}$  of rank $d+1$ and $(f,f')\in F$ such that
$$\phi((x,y)-(f,f'))=0.$$
Decompose $\phi=(A|B)$ with $A:E^{d_1+1}\to E^{d+1}$ and $B:E^{d_2+1} \to
E^{d+1}$. Then $$d_2+d_1+1={\rm{rk}\,\,}\phi\le {\rm{rk}\,\,} A+{\rm{rk}\,\,} B.$$
Note that ${\rm{rk}\,\,} A\le d_1+1$ and ${\rm{rk}\,\,} B\le d_2+1$ because of the number of columns. Then  one of the following cases occurs:
\begin{itemize}

\item[(1)] ${\rm{rk}\,\,} A=d_1$ or ${\rm{rk}\,\,} B=d_2$, 

\item[(2)]  ${\rm{rk}\,\,}A=d_1+1$ and ${\rm{rk}\,\,} B=d_2+1$.
\end{itemize}

${\rm{(1) }}$- If the rank of $B$ is $d_2$ then, with the   Gauss algorithm, one finds an invertible  matrix  $\Delta \in {\rm{Mat}}_{d+1}(\emor(E))$ such that 
\begin{equation*}
\Delta\phi=\left(
\begin{array}{cc}
\varphi_1&0\\
\star &\varphi_2
\end{array}
\right),
\end{equation*}
with $\varphi_1$ of rank $d_1+1$.

If the rank of $A$ is $d_1$ then  one finds an invertible matrix  $\Delta \in {\rm{Mat}}_{d+1}(\emor(E))$ such that 
\begin{equation*}
\Delta\phi=\left(
\begin{array}{cc}
\varphi_1&\star\\
0 &\varphi_2
\end{array}
\right),
\end{equation*}
with $\varphi_2$ of rank $d_2+1$.

Then either $x\in S_{d_1+1}(V_1,F)$ or $y\in S_{d_2+1}(V_2,F)$.
So $(x,y) \in \big( S_{d_1+1}(V_1,F) \times V_2 \big)
\cup \big(V_1\times S_{d_2+1}(V_2,F)\big)$.

${\rm(2)}$-  With
the Gauss algorithm one can find two invertible matrices $\Delta_i \in {\rm{Mat}}_{d+1}(\emor(E))$ such that 

\begin{equation*}
\begin{split}
\Delta_1\phi=&(aI_{d+1}|l)\\
\Delta_2\phi=&(l'|bI_{d+1})
\end{split}
\end{equation*}
with $a,b\in \emor(E){\setminus 0}$ and $I_{d+1}$ the identity matrix.
The last $d_2$ rows of $\Delta_1\phi$ show that $y\in S_{d_2}(V_2,F)$ and the first $d_1$ rows of 
$\Delta_2\phi$ show that $x\in S_{d_1}(V_1,F)$. Thus
$(x,y) \in \big(S_{d_1}(V_1,F)\times S_{d_2}(V_2,F)\big)$.\\

We now apply the inclusion (\ref{unisce}) to the case of curves, and we deduce a non-density result for surfaces.
Let $V_i=C_i$ be transverse curve in $E^2$.   By Theorem \ref{phill}, there exists $\varepsilon>0$ such that $S_1(C_i,\mathcal{O}_\varepsilon)$ has bounded height. In view of the  Bogomolov  Conjecture, a theorem of Ullmo, one can choose $\varepsilon$ such that  $S_2(C_i, \mathcal{O}_\varepsilon)$ is finite. Define  $F=\mathcal{O}_\varepsilon$. Then, relation (\ref{unisce}) implies  that  $S_{3}(C_1 \times C_2, \mathcal{O}_\varepsilon)$ has bounded height. In addition $C_1\times C_2$ is transverse in $E^4$. Using \cite{io2} Theorem 1.1, we conclude that  $S_{3}(C_1 \times C_2, \mathcal{O}_\varepsilon)$  is non-Zariski dense.

According to  Theorem \ref{phill}   and \cite{io2} Theorem 1.1, one can do similar considerations for hypersurfaces.

\end{example}

These last examples give evidence  that  the transverse or weak-transverse
hypothesis  is sufficient  for Conjecture \ref{hb}.  Precisely, the
idea is that if $U_1$ is a dense subset of $V_1$  of bounded
height, then  the set $U_1\times V_2$ is densely unbounded in
$V_1\times V_2$, (this is more or less what makes Property $(S)$
necessary for Theorem \ref{phill}). Instead if $U_1$ is Zariski closed in $V_1$, then the set $U_1\times V_2$ is still Zariski closed  in $V_1\times V_2$. 

Could one extend the idea  in the last examples to show that for the product of varieties satisfying Property $(S)$ Conjecture \ref{hb} holds? 

This is not an easy matter; even the case of $C_1\times C_2$ for $C_1$ transverse in $E^2$ and $C_2$ transverse in $E^3$ remains open.

%\begin{center}\begin{tabular}{l}

%Evelina Viada\\
%D-Math, \\
%R\"amistrasse 101\\
%8092 Zurich\\
%Suisse\\
%\texttt{viada@math.ethz.ch}
%\end{tabular}\end{center}

\end{document}